\documentclass[11pt]{article}

\usepackage[T1]{fontenc}
\usepackage{lmodern}
\usepackage{amsmath,amssymb,amsthm}
\usepackage[margin=1in]{geometry}
\usepackage[hidelinks]{hyperref}

\hypersetup{
  pdftitle={A linear gap for the maximal cp-rank},
  pdfauthor={Yair Lavi}
}

\numberwithin{equation}{section}

\newtheorem{theorem}{Theorem}[section]
\newtheorem{lemma}[theorem]{Lemma}
\newtheorem{proposition}[theorem]{Proposition}

\title{A linear gap for the maximal cp-rank}
\author{Yair Lavi}
\date{28 September 2026}

\begin{document}
\maketitle

\begin{abstract}

We prove that the maximal cp-rank $p_n$ satisfies $p_n\ge n(n-3)/2$ for odd $n\ge5$ and $p_n\ge n(n-3)/2-1$ for even $n\ge6$. Together with the known upper bound $p_n\le\binom{n+1}{2}-4$, this gives $p_n=n^2/2+O(n)$.
\end{abstract}



\section{Introduction}

A real symmetric $n\times n$ matrix $A$ is \textbf{completely positive} if
$A=BB^{\mathsf T}$ for an entrywise nonnegative real matrix $B$ with finitely
many columns. Its \textbf{cp-rank} $\operatorname{cpr}(A)$ is the least possible
number of columns of a matrix $B$ in such a representation; set
$\operatorname{cpr}(0)=0$. Write $\operatorname{CP}_n$ for the cone of these
matrices and

\begin{equation}
p_n=\max\{\operatorname{cpr}(A):A\in\operatorname{CP}_n\}.
\end{equation}

Write $\mathbb R_+^n=\{x\in\mathbb R^n:x_i\ge0\text{ for every }i\}$;
thus $x\ge0$ means that every coordinate of $x$ is nonnegative. A
real symmetric $n\times n$ matrix $S$ is \textbf{copositive} if
$x^{\mathsf T}Sx\ge0$ for every $x\in\mathbb R_+^n$.

Drew, Johnson, and Loewy conjectured $p_n=\lfloor n^2/4\rfloor$
for $n\ge4$ \cite{ref1}. Bomze, Schachinger, and Ullrich disproved this at
orders $7$--$11$ \cite{ref5} and at every order $n\ge12$ \cite{ref3}. The general
upper bound $p_n\le\binom{n+1}{2}-4$ for $n\ge5$ is \cite[Corollary~5.2]{ref2}.
Relative to this upper bound, the lower bounds of \cite{ref3} and
\cite[Theorem~7.4]{ref4} leave gaps of $O(n^{3/2})$ and
$O(n\log\log n)$, respectively.

We follow \cite[Sections~2--3]{ref5} and \cite[Lemmas~2.1--2.2 and Example~3.1]{ref3}
in using copositive matrices with nonzero vectors $u\ge0$ satisfying
$u^{\mathsf T}Su=0$ and in deleting the same coordinates from rows and
columns. Let $\widehat S$ be the order-seven matrix of \cite[Example~1]{ref5}.
With zero-based indices modulo $7$, the choice
$t=4\pi-7\arccos(-1/6)$ in our family gives
$S_{ij}=\widehat S_{2i,2j}/162$. Here we give
a uniform family in every odd order, prove its copositivity, classify
all such vectors up to positive scaling, and obtain explicit lower
bounds in both parities without finite-order computation.

\begin{theorem}
For every integer $n\ge5$, put
$\varepsilon_n=0$ if $n$ is odd and $\varepsilon_n=1$ if $n$ is even. Then

\begin{equation}
\boxed{\frac{n(n-3)}2-\varepsilon_n
\ \le\ p_n\ \le\ \binom{n+1}{2}-4.}
\end{equation}
\end{theorem}

Thus $p_n=\tfrac12n^2+O(n)$. The difference between the upper and lower
expressions in (1.2) is $2n-4+\varepsilon_n$. The construction gives
the lower expression on a particular subset of $\operatorname{CP}_n$;
(1.2) does not assert an exact value of $p_n$.

At $n=5$, our bound $5$ compares with $p_5=6$ \cite[Remark~5.1]{ref2}.
At $n=6$, our bound $8$ compares with $9\le p_6\le15$
\cite[Table~2]{ref3} and \cite[Theorem~6.1]{ref2}. At $n=7,13$ our bounds
match published lower bounds $14,65$ \cite[Example~1]{ref5} and \cite[Table~1]{ref3};
at $n=14,15$ they give $76,90$, below $80,95$ \cite[Table~1]{ref3}.
For $n=8,\ldots,12$, our values $19,27,34,44,53$ exceed the
published $18,26,30,37,50$ \cite[Examples~2--3]{ref5}, \cite[Table~1]{ref3}, and \cite[Table~1]{ref4}.

The proof classifies the finitely many nonnegative vectors $u\ne0$ with
$u^{\mathsf T}Su=0$ for a suitable copositive $S$, up to positive scaling.
Every nonzero factor column of a completely positive matrix orthogonal
to $S$ is a positive multiple of one of them. Lemma 2.1 shows that the
dimension of the span of their outer products is the maximum cp-rank
among these matrices; Sections 3--5 compute it in odd and even order.

\section{Finite zero sets and cp-rank}

Let $S$ be any real symmetric $n\times n$ matrix. A nonzero vector
$u\in\mathbb R_+^n$ with $u^{\mathsf T}Su=0$ is a \textbf{nonnegative zero} of
$S$. Its \textbf{support} is $\operatorname{supp}(u)=\{i:u_i>0\}$. The ray
generated by a nonnegative zero $u$ is $\{cu:c\ge0\}$; two nonnegative
zeros determine the same ray exactly when one is a positive multiple of
the other. We call such a ray a \textbf{zero ray} of $S$.

On the space of real symmetric $n\times n$ matrices, we use the
\textbf{Frobenius inner product}
$\langle S,A\rangle=\operatorname{tr}(SA)$. An \textbf{exposed face} of a cone
$K$ is its intersection with the zero set of a linear functional that
is nonnegative on $K$. If $S$ is copositive, then for every
$A=BB^{\mathsf T}\in\operatorname{CP}_n$, writing the columns of $B$ as
$b_1,\ldots,b_s$ gives
$\langle S,A\rangle=\sum_{j=1}^s b_j^{\mathsf T}Sb_j\ge0$.
Hence the following set is an exposed face of the
completely positive cone $\operatorname{CP}_n$:

\begin{equation}
F(S)=\{A\in\operatorname{CP}_n:\langle S,A\rangle=0\}.
\end{equation}

The following finite-zero principle appears in
\cite[Lemma~2.1]{ref5} and \cite[Lemmas~2.1--2.2]{ref3}.

\begin{lemma}[finite-zero face]
Suppose a copositive $S$ has exactly
$q<\infty$ nonnegative zero rays, represented by
$u^{(1)},\ldots,u^{(q)}$. If

\begin{equation}
h=\dim\operatorname{span}\{u^{(a)}u^{(a)\mathsf T}:1\le a\le q\},
\end{equation}

then the maximum cp-rank of matrices in $F(S)$ is $h$.

\end{lemma}

\begin{proof}
If $q=0$, then $F(S)=\{0\}$ and $h=0$. Otherwise, choose a
factorization $A=BB^{\mathsf T}$ for $A\in F(S)$. Copositivity makes
every summand's pairing with $S$ nonnegative, so each nonzero column of
$B$ lies on a listed zero ray. Thus $F(S)$ is generated by their outer
products. The
conic form of Carath\'eodory's theorem says that any point of a cone
spanning an $h$-dimensional space is a nonnegative combination of at
most $h$ generators, so cp-rank is at most $h$. A point of $F(S)$
outside the finitely many spans of fewer than $h$ generators has
cp-rank $h$. Such a point exists because a finite union of proper
linear subspaces cannot cover an $h$-dimensional cone.

\end{proof}

\section{A copositive family in odd order}

Let $n=2m+1\ge5$ and index coordinates by $\mathbb Z/n\mathbb Z$ in
cyclic order. We also regard index $i$ as the vertex at polar angle
$2\pi i/n$ of the unit regular $n$-gon. The \textbf{cyclic distance} $d(i,j)$
is the smaller number of steps along the polygon from $i$ to $j$, so
$1\le d(i,j)\le m$ for $i\ne j$. A \textbf{cyclic interval of $L$ steps} is
$\{i,i+1,\ldots,i+L\}$ modulo $n$, where $0\le L<n$. Fix

\begin{equation}
0<t<\frac{2\pi}{2n-3},\qquad
a=\frac{2\pi+3t}{n},\qquad
\phi_k=ak-t\quad(1\le k\le m),
\end{equation}

and define a real symmetric matrix $S$ by

\begin{equation}
S_{ii}=1,\qquad S_{ij}=\cos\phi_{d(i,j)}\quad(i\ne j).
\end{equation}

At $n=5$, the matrices in (3.2) belong, up to a permutation of
coordinates, to Hildebrand's order-five copositive family \cite{ref6}.

This matrix is \textbf{circulant}: its entries satisfy
$S_{i+h,j+h}=S_{ij}$ for every $h\in\mathbb Z/n\mathbb Z$.

The interval in (3.1) gives $0<\phi_1<\cdots<\phi_m<\pi$.
Indeed $\phi_1=(2\pi-(n-3)t)/n>0$ and
$\phi_m=\pi-\pi/n+(n-3)t/(2n)<\pi$.
The \textbf{pair angle} of $\{i,j\}$ is $\phi_{d(i,j)}$. For an index set
$I$, the \textbf{principal submatrix} on $I$ is
$S_I=(S_{ij})_{i,j\in I}$.

For three distinct vertices $v_1,v_2,v_3$ in cyclic order, choose
representatives $0\le v_1<v_2<v_3<n$. Their \textbf{cyclic gaps} are
$k=v_2-v_1$, $\ell=v_3-v_2$, and $r=n+v_1-v_3$: the positive numbers
of steps from $v_1$ to $v_2$, from $v_2$ to $v_3$, and from $v_3$ back
to $v_1$. They satisfy $k+\ell+r=n$; starting the cyclic listing at a
different vertex only permutes them cyclically. We call the triple
\textbf{surrounding} if $k,\ell,r\le m$. Then each gap is the cyclic
distance of its endpoint pair, and the three pair angles sum to
$a(k+\ell+r)-3t=an-3t=2\pi$.

For example, let $n=7$, $m=3$, $t=\pi/9$ (which satisfies (3.1)), and
therefore $a=\pi/3$.
The triple $\{0,2,4\}$ has gaps $(k,\ell,r)=(2,2,3)$, so it is
surrounding. Its pair distances are $d(0,2)=d(2,4)=2$ and $d(4,0)=3$;
the pairs $\{0,2\}$ and $\{2,4\}$ each have angle $\phi_2=5\pi/9$,
while $\{0,4\}$ has angle $\phi_3=8\pi/9$. These three angles sum to
$2\pi$.

\begin{lemma}[triangle zeros]
Let $T=\{v_1,v_2,v_3\}$ be a
surrounding triple, with its vertices listed in cyclic order and gaps
$k,\ell,r$ as above. Up to multiplication by a positive scalar, there
is exactly one nonnegative zero $u$ of $S$ with
$\operatorname{supp}(u)=T$. A representative, with all other
coordinates zero, is

\begin{equation}
u_{v_1}=\sin\phi_\ell,\qquad
u_{v_2}=\sin\phi_r,\qquad
u_{v_3}=\sin\phi_k,
\end{equation}

\end{lemma}

\begin{proof}
For the vertices in cyclic order, take planar unit vectors
$w_1,w_2,w_3$ at angles $0,\phi_k,\phi_k+\phi_\ell$.
Their \textbf{Gram matrix}, the matrix $(w_a\cdot w_b)_{a,b=1}^3$ of pairwise
inner products, is $S_T$, and
$\sin\phi_\ell w_1+\sin\phi_r w_2+\sin\phi_k w_3=0$:
the vertical coordinate is
$\sin\phi_k(\sin\phi_r+\sin(\phi_k+\phi_\ell))=0$,
and the horizontal coordinate is
$\sin\phi_\ell+\sin(\phi_k+\phi_r)=0$, using
$\phi_k+\phi_\ell+\phi_r=2\pi$.
Thus (3.3) is the unique nonnegative kernel ray of the rank-two
positive semidefinite matrix $S_T$.
\end{proof}

In the seven-vertex example, (3.3) gives
$u_0=u_4=\sin(5\pi/9)$ and $u_2=\sin(8\pi/9)$, with all other
coordinates zero.

We now prove that $S$ is copositive and has no other zeros. An \textbf{open
semicircle} means an open arc of angular length $\pi$ on the unit
circle. The next lemma concerns the regular $n$-gon defined above.

\begin{lemma}[four-vertex geometry]
A vertex set containing no
surrounding triple lies in an open semicircle, hence in a cyclic interval
of at most $m$ steps. A four-vertex set containing a surrounding triple
contains exactly two such triples. They share a pair $\{i,j\}$.

\end{lemma}

\begin{proof}
A surrounding triple is precisely a triple of polygon vertices
whose convex hull contains the origin: three points on a circle enclose
its center exactly when they do not fit in an open semicircle, equivalently
when none of their cyclic gaps exceeds $m$. If the origin were in the convex
hull of a vertex set, planar Carath\'eodory's theorem would put it in the
convex hull of at most three of those vertices. One vertex cannot
contain the origin, and a segment between two vertices could contain
it only if they were antipodal, impossible for odd $n$. Hence a set
with no surrounding triple has a convex hull omitting the origin.
A linear functional strictly positive on this convex hull places all
its vertices in an open half-plane through the origin, hence in an open
semicircle spanning at most $m$ cyclic steps.

For four vertices let $g_1,g_2,g_3,g_4$ be their cyclic gaps, summing
to $2m+1$. If any $g_i>m$, no triple is surrounding. Otherwise the
triple obtained by omitting a vertex is surrounding exactly when the
two gaps merged by that omission sum to at most $m$. The sums
$g_1+g_2$ and $g_3+g_4$ add to $2m+1$, so exactly one is at most $m$;
the same holds for $g_2+g_3$ and $g_4+g_1$. Hence exactly two triples
are surrounding. They share a pair.
\end{proof}

\begin{lemma}[overlapping triangle zeros]
If $u,v$ represent the rays of the two
surrounding triples in a four-set, then $u^{\mathsf T}Sv>0$.

\end{lemma}

\begin{proof}
Because $S$ depends only on cyclic distance, rotations and
reflections preserve it. Relabel the common pair as $\{0,k\}$ with
$1\le k\le m$. A third vertex of a surrounding triple through this
pair lies in $\{m+1,\ldots,m+k\}$. Write the third vertices of
$u,v$ as $p,q$, respectively, and set
$d=d(p,q)=|p-q|$, so $1\le d\le m-1$.
Since $1\le p-k,n-p\le m$, we have
$ap-2t=\phi_k+\phi_{p-k}=2\pi-\phi_{n-p}$, and likewise for $q$.
Choose planar unit vectors at angles
$0,\phi_k,ap-2t,aq-2t$, indexed by $0,k,p,q$, and let $G$ be their
Gram matrix. Since $an=2\pi+3t$, its entries agree with $S$ on each
surrounding triple, while $G_{pq}=\cos(ad)$. Each of $u,v$ is a
dependence among its triple's planar vectors, so $u^{\mathsf T}Gv=0$.
The only entry at which $S$ and $G$ can differ in $u^{\mathsf T}Sv$
is the pair $\{p,q\}$:

\begin{equation}
u^{\mathsf T}Sv=u_pv_q[\cos(ad-t)-\cos(ad)]
=2u_pv_q\sin(ad-t/2)\sin(t/2)>0.
\end{equation}

Here $ad-t/2>0$ follows from $a>t$, and $ad<\pi$ follows from
$d\le m-1$ and $t<2\pi/(n-3)$, a weaker consequence of (3.1).
\end{proof}

\begin{proposition}[complete zero list]
The matrix $S$ is copositive,
and its nonnegative zero rays are exactly those in Lemma 3.1.

\end{proposition}

\begin{proof}
Minimize $x^{\mathsf T}Sx$ on the simplex
$\Delta=\{x\ge0:\sum_i x_i=1\}$, and let $I$ be the support of a
minimizer. If $I$ contains no surrounding triple, Lemma 3.2 places it
in a cyclic interval $\{s,s+1,\ldots,s+L\}$ modulo $n$, with $L\le m$.
For each $i\in I$, choose the unique $r_i\in\{0,\ldots,L\}$ with
$i\equiv s+r_i\pmod n$, and let
$G_{ij}=\cos(a(r_i-r_j))$, a planar Gram matrix. For $i\ne j$ in
the interval, $d=|r_i-r_j|\le m$ and

\begin{equation}
S_{ij}-G_{ij}=2\sin(ad-t/2)\sin(t/2)>0.
\end{equation}

Indeed $a>t$ and
$am-t/2=\pi-\pi/n+(2n-3)t/(2n)<\pi$. The diagonals agree.
Since $G$ is positive semidefinite, the off-diagonal correction gives
$x^{\mathsf T}Sx>0$ for supports of size at least two;
singleton supports have value one.

If $I$ is a surrounding triple, $S_I$ is positive semidefinite by
Lemma 3.1. If $I$ contains a surrounding triple and at least one more
vertex, choose a four-set inside $I$ containing it. Let $u,v$ be its two
triangle rays, each normalized to have coordinate sum one. Then
$y=u-v$ has coordinate sum zero, and Lemma 3.3 gives
$y^{\mathsf T}Sy=-2u^{\mathsf T}Sv<0$.
Both $x+\epsilon y$ and $x-\epsilon y$ remain in $\Delta$ for small
$\epsilon>0$, but the average of their quadratic values is
$x^{\mathsf T}Sx+\epsilon^2 y^{\mathsf T}Sy<x^{\mathsf T}Sx$,
contradicting minimality. Hence the minimum is nonnegative and $S$ is
copositive.

A normalized zero minimizes on $\Delta$, so the cases force its
support to be a surrounding triple; $S_I$ then forces (3.3).
\end{proof}

\section{The odd-order face dimension}

Index the rows of $M$ and $H$ by surrounding triples $T$. Index the
columns of the zero--one matrix $M$ by unordered pairs
$e=\{i,j\}$ with $0\le i<j<n$, and set
$M_{T,e}=1$ if $e\subset T$ and $M_{T,e}=0$ otherwise. For each $T$,
let $u^{(T)}$ be the representative in (3.3), with
$u_i^{(T)}=0$ for $i\notin T$. Index the columns of $H$ by
$(i,j)$ with $0\le i\le j<n$ and define

\[
H_{T,(i,j)}=u_i^{(T)}u_j^{(T)}\qquad(0\le i\le j<n).
\]

Thus $(i,i)$ labels a diagonal coordinate, while for $i<j$ the column
$(i,j)$ of $H$ corresponds to the column $\{i,j\}$ of $M$.
The rank of $H$ is the dimension $h$ in (2.2).

\begin{lemma}[rank transfer]
$\operatorname{rank}H=\operatorname{rank}M$.

\end{lemma}

\begin{proof}
For a triple $T=\{i,j,k\}$, let
$g_T=\sin\phi_{d(i,j)}\sin\phi_{d(i,k)}\sin\phi_{d(j,k)}>0$.
For a pair $e=\{a,b\}$ write $d(e)=d(a,b)$, and let $H_{\mathrm{off}}$
be the block of $H$ with only off-diagonal columns. If $e\subset T$,
equation (3.3) gives $u_au_b\sin\phi_{d(e)}=g_T$. Thus

\begin{equation}
H_{\mathrm{off}}=\operatorname{diag}(g_T)\,M\,
\operatorname{diag}\bigl(1/\sin\phi_{d(e)}\bigr),
\end{equation}

Both diagonal factors are invertible, so this block has rank
$\operatorname{rank}M$. For every zero $u$ and vertex $i$ with
$u_i>0$, varying that coordinate slightly in both directions preserves
nonnegativity. Copositivity and $u^{\mathsf T}Su=0$ therefore give
$(Su)_i=0$. If $u_i=0$, the following equation is trivial. Hence

\begin{equation}
u_i^2=-\sum_{j\ne i}S_{ij}u_iu_j.
\end{equation}

Each diagonal column of $H$ is a fixed linear combination of its
off-diagonal columns. It adds no rank.
\end{proof}

\begin{lemma}[cyclic incidence rank]
For every odd $n=2m+1\ge5$,

\begin{equation}
\operatorname{rank}M=\binom n2-n=\frac{n(n-3)}2.
\end{equation}

\end{lemma}

\begin{proof}
Because $n$ is odd, every unordered pair has a unique form
$\{i,i+k\}$ with $1\le k\le m$. Label it by $(i,k)$ and write
$x_{i,k}$ for its coordinate.
For a surrounding triple with directed cyclic gaps $k,\ell,r\le m$
and $k+\ell+r=n$, the equation $Mx=0$ is

\begin{equation}
x_{i,k}+x_{i+k,\ell}+x_{i+k+\ell,r}=0.
\end{equation}

Over $\mathbb C$, the functions $i\mapsto z^i$ for $z^n=1$ form a
basis of functions on $\mathbb Z/n\mathbb Z$. Thus every array $x$ has
the Fourier expansion

\[
x_{i,k}=\sum_{z^n=1}z^if_k(z),\qquad
f_k(z)=\frac1n\sum_{i=0}^{n-1}z^{-i}x_{i,k}.
\]

Substitution into (4.4) and comparison of the independent functions
$i\mapsto z^i$ show that the complex kernel splits into \textbf{modes}
$x_{i,k}=z^if_k$, one for each $z$. The equations for a mode are

\begin{equation}
f_k+z^kf_\ell+z^{k+\ell}f_r=0
\qquad(k+\ell+r=n;\ 1\le k,\ell,r\le m).
\end{equation}

For $z\ne1$, subtract the equations for $(k,m+1-k,m)$ and
$(k,m,m+1-k)$; the result is (4.6).

\begin{equation}
(1-z^m)f_{m+1-k}=(1-z^{m+1-k})f_m.
\end{equation}

Since $\gcd(m,n)=1$, $z^m\ne1$. As $k$ varies, (4.6) fixes every
$f_j$ as a multiple of $1-z^j$. Conversely these values satisfy
(4.5), whose left side telescopes to $1-z^n=0$. Thus the kernel in
each nontrivial mode has dimension one.

For $z=1$, compare the equations for
$(k,m+1-k,m)$ and $(k+1,m+1-k,m-1)$,
$1\le k\le m-1$. They give
$f_{k+1}-f_k=f_m-f_{m-1}$, so $f_k=Ak+B$. Every equation (4.5) then
reduces to $nA+3B=0$. This mode also has dimension one. There are $n$
values of $z$, so $\dim_{\mathbb C}\ker M=n$. The matrix $M$ is
real, hence its real nullity is $n$ as well. With $\binom n2$ columns,
(4.3) follows.
\end{proof}

By Proposition 3.4, Lemmas 2.1, 4.1, and 4.2, the exposed face
$F(S)\subseteq\operatorname{CP}_n$ has
maximum cp-rank exactly $n(n-3)/2$. This proves the odd case of the
lower bound in Theorem 1.1.

\section{One-coordinate restriction and even order}

Let $N=2m\ge6$, let $n=N+1=2m+1$, and form the odd-order matrix
$S$ in (3.2). Delete coordinate $0$ to obtain an $N\times N$
principal submatrix $S'$. It is copositive because a nonnegative
vector on the surviving coordinates extends by zero. Such a vector is
a zero of $S'$ exactly when its zero extension is a zero of $S$.
Proposition 3.4 therefore says that the zero rays of $S'$ are precisely
the surrounding-triple rays avoiding $0$.

Let $M'$ be the submatrix of $M$ with rows for surrounding triples
avoiding $0$ and columns for pairs of surviving vertices.
The rank transfer in Lemma 4.1 applies to this
sublist: equation (4.1) restricts to its rows and surviving pairs, and
(4.2) uses only surviving coordinates. We calculate $\operatorname{rank}M'$.

\begin{proposition}[restriction rank]
For every $m\ge3$,

\begin{equation}
\operatorname{rank}M'=\binom{2m}{2}-2m-1
=\frac{N(N-3)}2-1.
\end{equation}

\end{proposition}

\begin{proof}
Divide the columns of the full odd-order matrix $M$ into
$P_0$, the pairs avoiding $0$, and $P_\bullet$, the pairs through $0$.
Divide its rows into those avoiding $0$, whose $P_0$ block is $M'$,
and the rows $R_0$ of surrounding triples through $0$. Write $U$ for
the row space of $M'$
padded by zeros on $P_\bullet$, and $W$ for the row space of $R_0$.
Since $\operatorname{rowspace}M=U+W$,

\begin{equation}
\operatorname{rank}M'=\operatorname{rank}M-\operatorname{rank}R_0
+\dim(U\cap W).
\end{equation}

A triple $\{0,j,k\}$ with $1\le j<k\le2m$ is surrounding exactly when
$j\le m<k$ and $k-j\le m$. Encode these triples by the edges of a
bipartite graph $B$ with left vertices $1,\ldots,m$ and right
vertices $m+1,\ldots,2m$, where $j$ is adjacent to
$m+1,\ldots,j+m$. Write $E(B)$ for its edge set. The graph is
connected and has

\begin{equation}
|E(B)|=\sum_{j=1}^m j=\frac{m(m+1)}2.
\end{equation}

Each row of $R_0$ contains the $P_0$ coordinate indexed by its pair
$\{j,k\}$, which occurs in no other row of $R_0$. Thus its rows are
independent and $\operatorname{rank}R_0=|E(B)|$.
For a pair $\{a,b\}$, let $b_{ab}$ be its unit coordinate vector in
the space indexed by all unordered pairs. An element of $W$ is
$\sum_{jk\in E(B)}\lambda_{jk}(b_{0j}+b_{0k}+b_{jk})$. Its
$P_\bullet$ coordinates vanish precisely when
$\sum_{f\ni v}\lambda_f=0$ at each vertex $v$ of $B$. Define
$Z(B)=\{\lambda\in\mathbb R^{E(B)}:\sum_{f\ni v}\lambda_f=0
\text{ for every vertex }v\}$, and use the same definition for
subgraphs $G$ of $B$. For an edge weighting $\lambda\in\mathbb R^{E(G)}$,
its support is $\{e\in E(G):\lambda_e\ne0\}$. Orient each edge of $G$
from left to right, and let $D_G$ be its vertex-edge matrix, with
$-1$ at the left endpoint, $+1$ at the right endpoint, and zero
elsewhere. The equations $D_G\lambda=0$ differ from the defining
vertex-sum equations only by signs, so $Z(G)=\ker D_G$ is the real
cycle space of $G$. The map
$\pi(\lambda)=\sum_{jk\in E(B)}\lambda_{jk}b_{jk}$ is injective, so
$\dim(U\cap W)$ equals the dimension of the $\lambda\in Z(B)$ for
which $\pi(\lambda)\in U$.

The edge $e_*=\{m,m+1\}$ belongs to $B$ but its column in $M'$ is
zero. Indeed no surrounding triple avoiding $0$ contains both
$m$ and $m+1$: a third vertex on either complementary arc makes one
cyclic gap exceed $m$. Hence $\pi(\lambda)\in U$ implies
$\lambda_{e_*}=0$. Let $B-e_*$ be the graph obtained by deleting
this edge, and regard $Z(B-e_*)$ as a subspace of $\mathbb R^{E(B)}$
by assigning zero to its $e_*$ coordinate. This gives the upper
inclusion

\begin{equation}
\{\lambda\in Z(B):\pi(\lambda)\in U\}\subseteq Z(B-e_*),
\end{equation}

We prove the reverse inclusion by four-cycles. For any cycle in a
bipartite graph, its \textbf{cycle vector} assigns alternating $+1$ and
$-1$ to successive edges and zero to all other edges. A four-cycle in $B$
has vertices $j<j'\le m<k<k'$ with $k'-j\le m$ and alternating
pair-coordinate vector, the image of its cycle vector under $\pi$,
$v=b_{jk}-b_{j'k}+b_{j'k'}-b_{jk'}$.
It avoids $e_*$ exactly when $j'\le m-1$ or $k\ge m+2$. In the first
case set $\ell=j'+m+1>k'$; in the second set
$\ell=k-m-1<j$. The defining inequalities of $B$ show in either
case that all four triples $\{j,k,\ell\}$,
$\{j',k,\ell\}$, $\{j',k',\ell\}$, and
$\{j,k',\ell\}$ are surrounding and avoid $0$. Indeed, for
$x\in\{j,j'\}$ and $y\in\{k,k'\}$ the cyclic gaps of
$\{x,y,\ell\}$ are, in the first case,
$y-x\le m$, $\ell-y\le m$, $n-\ell+x=m+x-j'\le m$;
in the second case they are $x-\ell\le m$, $y-x\le m$,
$n-y+\ell=m+k-y\le m$. If $T_{xyz}$
denotes the incidence row of $\{x,y,z\}$, cancellation of the pairs through
$\ell$ gives

\begin{equation}
T_{jk\ell}-T_{j'k\ell}+T_{j'k'\ell}-T_{jk'\ell}=v.
\end{equation}

Thus every four-cycle vector of $B-e_*$ maps into $U$.

These four-cycles span $Z(B-e_*)$. For any nonzero $\lambda$ in this
space, a vertex incident with exactly one edge in its support would
violate the vertex-sum equation. Thus every vertex incident with a
support edge has degree at least two there, and the finite support
subgraph contains a cycle. Subtract a multiple of that
cycle's vector to eliminate one support edge, without adding edges to
the support; finitely many repetitions show that cycle vectors span
the space. Now consider a cycle $C$ of length at least six.
Choose its largest left vertex $j_*$ other
than $m$. Every right vertex $k$ on $C$ has a cycle-neighbor
$j_a\ne m$, so $k\le j_a+m\le j_*+m$. Thus $j_*$ is adjacent in
$B-e_*$ to every right vertex of $C$; this edge cannot be $e_*$, since
$j_*\ne m$. There are at least three right vertices but only two
cycle-neighbors of $j_*$, so $C$ has a chord, an edge joining two
nonconsecutive vertices of $C$. It divides $C$ into two shorter cycles.
Choose the signs of their alternating cycle vectors so their chord
entries cancel; their sum is the cycle vector of $C$. Repeating this
split expresses that vector as a sum of four-cycle vectors. Also,
$B-e_*$ remains connected when $m\ge3$: $m$ meets
every right vertex except $m+1$, the latter meets every left vertex
except $m$, and $m,m+1$ are joined by the path
$m,(m+2),(m-1),(m+1)$. For $G=B-e_*$, the equation
$D_G^{\mathsf T}c=0$ makes the vertex labels $c$ equal across every
edge. Connectedness makes $c$ constant, so $\operatorname{rank}D_G=2m-1$
and

\begin{equation}
\dim Z(B-e_*)=\bigl(|E(B)|-1\bigr)-2m+1
=\frac{m(m-3)}2.
\end{equation}

Equations (5.4)--(5.6) show that $\dim(U\cap W)=m(m-3)/2$.
Substitute this, (5.3), and
$\operatorname{rank}M=\binom{2m+1}{2}-(2m+1)$ from Lemma 4.2 into
(5.2). The result is
$\operatorname{rank}M'=2m^2-3m-1$, which is (5.1).
\end{proof}

\begin{proof}[Proof of Theorem 1.1]
The odd lower bound follows after Lemma 4.2.
For even $n=N\ge6$, the rank transfer of Lemma 4.1 applies to the
surviving zero rays, so Proposition 5.1 and Lemma 2.1 give a matrix
in the exposed face $F(S')\subseteq\operatorname{CP}_N$ of cp-rank
$N(N-3)/2-1$. The upper bound is
\cite[Corollary~5.2]{ref2}. Subtracting the two expressions in (1.2)
gives $2n-4+\varepsilon_n$.
\end{proof}

\section*{Acknowledgments}

The proofs presented in this paper were carried out by GPT-5.6-sol,
GPT-6 Astra, Claude Fable 5, and Claude Fable 5.1, under the guidance
of the author. The author has reviewed the resulting proof arguments.
Responsibility for the final text rests with the author.

\end{document}